\documentclass[oupthm]{CUP-JNL-BCM}%

\usepackage{graphicx}
\usepackage{multicol,multirow}
\usepackage{amsmath,amssymb,amsfonts}
\usepackage{mathrsfs}
\usepackage{rotating}
\usepackage{appendix}
\usepackage[numbers]{natbib}

\theoremstyle{oupplain}
\newtheorem{theorem}{Theorem}[section]
\newtheorem{lemma}[theorem]{Lemma}
\theoremstyle{oupdefinition}
\newtheorem{definition}{Definition}[section]
\theoremstyle{oupremark}
\newtheorem{remark}[theorem]{Remark}
\newtheorem{example}[theorem]{Example}
\theoremstyle{oupproof}
\newtheorem{proof}{Proof}

\numberwithin{equation}{section}

\articletype{RESEARCH ARTICLE}
\jyear{2025}

\begin{document}

\begin{Frontmatter}

\title[Spectral Bounds for Tensors Derived from Trace Functionals and Wasserstein Distance in Tensor Spaces] {Spectral Bounds for Tensors Derived from Trace Functionals and Wasserstein Distance in Tensor Spaces}

\author{Hemant Sharma$^1$}
\author{Nachiketa Mishra$^2$\thanks{Correspondence to : nmishra@iiitdm.ac.in}}

\authormark{H. Sharma and N. Mishra}

\address{\orgname{$^1$Department of Mathematics
        Indian Institute of Information Technology, Design and Manufacturing Kancheepuram}, \orgaddress{\city{Chennai}, \country{ India}}
\email{sharmahemant39@gmail.com}}
\address{\orgname{$^2$Department of Mathematics
        Indian Institute of Information Technology, Design and Manufacturing Kancheepuram}, \orgaddress{\city{Chennai}, \country{ India}}
\email{nmishra@iiitdm.ac.in}}



\abstract{This article introduces a trace-based metric on the space of positive semi-definite (PSD) tensors, offering a geometric perspective that connects their algebraic structure to their intrinsic geometric properties. It defines the Bures-Wasserstein distance on tensor spaces, establishing clear measurements between tensors. Moreover, the study derives trace-based eigenvalue bounds for PSD tensors and analyzes how these bounds depend on the PSD condition. The behavior of these bounds is further explored when the PSD requirement is relaxed, with illustrative examples provided to support the theoretical findings. In addition, a detailed complexity analysis is carried out for the methods proposed in this study.}

\end{Frontmatter}

\section{Introduction}
\emph{Motivation:~}Tensors, as natural extensions of matrices, have gained significant importance in various fields, including machine learning \cite{machinelearn}, signal processing \cite{signalprocess}, quantum computing \cite{quantam}, and numerical optimization \cite{numerical}. In particular, the study of positive semi-definite (PSD) tensors is crucial for generalizing spectral properties and geometric structures that are well understood in matrix analysis. The eigenvalue bounds for PSD matrices have numerous applications in stability analysis, optimization, and graph theory. However, extending these results to higher-order tensors presents new mathematical and computational challenges.\\[-1ex]

\noindent One of the most structured and computationally efficient ways to extend matrix operations to tensors is through the T-product \cite{kathrinLund} structure, which provides an approach for defining multiplication, decomposition, and spectral properties of tensors. The T-product, utilizing the discrete Fourier transform (DFT), facilitates tensor computations while retaining essential matrix properties, making it a valuable approach in tensor analysis.\\[-1ex]

\noindent This article focuses on establishing trace-based eigenvalue bounds for PSD tensors under the T-product. Furthermore, we extend fundamental matrix distance metrics, including the Bures-Wasserstein distance \cite{bures}, to tensor spaces using the T-product structure. A new trace-based distance metric is introduced to provide a geometric structure for PSD tensors, bridging algebraic and geometric properties.\\[-1ex]

\noindent \emph{Related work and problem:~}\\The study of PSD matrices has been well developed in classical linear algebra, with well-known results on eigenvalue bounds, trace-based inequalities, and spectral distance metrics. In matrix analysis, the Bures-Wasserstein distances serve as key measures for comparing PSD matrices and have widespread applications in quantum information \cite{quantaminform} and optimal transport \cite{optimal}. Extending these tools to higher-order tensors presents a challenge due to the absence of a well-defined framework for matrix-like operations in tensor spaces\\[-1ex]

\noindent Several approaches have been proposed to generalize matrix operations to tensors, including mode-wise products \cite{tensormmode}, Tucker decomposition, and tensor norms \cite{numerical}. However, these frameworks often fail to preserve fundamental algebraic properties such as associativity and invertibility, which are essential in spectral analysis and optimization.\\[-1ex]

\noindent One of the primary challenges in extending trace-based eigenvalue bounds and distance metrics to PSD tensors lies in establishing a well-defined notion of trace and eigenvalues within the T-product framework. Since the trace of a tensor is not as straightforward as in matrices, constructing trace-based metrics that retain useful geometric interpretations requires a careful formulation. This article addresses these challenges by:
\begin{itemize}
    \item Defining trace-based eigenvalue bounds for PSD tensors under the T-product.
    \item Extending classical distance metrics (Bures-Wasserstein) to tensor spaces using T-product operations.
    \item Introducing a trace-based distance metric that provides a geometric structure for PSD tensors.
\end{itemize}

\noindent These contributions provide a more structured understanding of PSD tensors and their applications in optimization, signal processing, and machine learning.\\[-1ex]

\noindent \emph{Contributions:~}\\This article makes several key contributions to the study of positive semi-definite (PSD) tensors under the T-product framework. First, it establishes trace-based eigenvalue bounds for PSD tensors and examines their dependence on the PSD condition, analyzing their behavior when this assumption is relaxed. Theoretical findings are further supported by illustrative examples. Second, it extends the Bures-Wasserstein distance widely used in matrix analysis to tensor spaces using the T-product formulation, providing a structured approach to comparing PSD tensors while preserving essential spectral properties.\\[-1ex]

\noindent Additionally, a trace-based distance metric is introduced, designed using trace operations within the T-product structure, which naturally connects the algebraic structure of tensors to their geometric interpretation. Lastly, numerical examples are presented to demonstrate the performance of the proposed eigenvalue bounds and distance metric, providing theoretical confirmation and practical insights into their applicability.\\

\noindent The structure of the article is as follows: Section \ref{def} establishes the foundation by introducing essential notation, definitions, and preliminary results that serve as the basis for subsequent discussions. Section \ref{results} explores trace bounds, presenting several eigenvalue inequalities for PSD tensors and their products, and analyzing the impact of the PSD condition on these bounds. Additionally, examples are provided to illustrate the theorems. Section \ref{distance} introduces the concept of tensor distance, formulated using the Bures-Wasserstein metric on the space of PSD tensors. Finally, Section \ref{conclusion} outlines the main findings and offers concluding observations.

\section{Definitions and Preliminaries}\label{def}

The T-product is a multiplication operation defined for third-order tensors. A third-order tensor  
\[
\mathcal{A} = \left[a_{i j k}\right], \quad 1 \leq i \leq m, \, 1 \leq j \leq n, \, 1 \leq k \leq p,
\]
is an array consisting of \(m \times n \times p\) entries. The set of all such tensors over the real or complex fields is denoted by \(\mathbb{R}^{m \times n \times p}\) or \(\mathbb{C}^{m \times n \times p}\), respectively. For any tensor \(\mathcal{A} \in \mathbb{R}^{m \times n \times p}\), its \textbf{frontal slices} are denoted as \(A^{(k)} = \mathcal{A}(:, :, k) \in \mathbb{R}^{m \times n}\) for \(k = 1, 2, \ldots, p\).\\  

\noindent To facilitate the T-product operation, the following linear operators are introduced: \textbf{bcirc}, \textbf{unfold}, and \textbf{fold}, as described in \cite{kathrinLund}.  

\begin{enumerate}
    \item The \textbf{block circulant matrix} of \(\mathcal{A}\), denoted as \(\operatorname{bcirc}(\mathcal{A})\), is defined as:  
    \begin{equation}\label{bcircdef}
    \operatorname{bcirc}(\mathcal{A}) := \begin{bmatrix}
    A^{(1)} & A^{(p)} & A^{(p-1)} & \cdots & A^{(2)} \\
    A^{(2)} & A^{(1)} & A^{(p)} & \cdots & A^{(3)} \\
    \vdots & \vdots & \ddots & \vdots & \vdots \\
    A^{(p)} & A^{(p-1)} & \cdots & A^{(2)} & A^{(1)}
    \end{bmatrix}.
    \end{equation}  

    \item The \textbf{unfolding} operator rearranges the tensor \(\mathcal{A}\) into a tall matrix as follows:  
    \[
    \operatorname{unfold}(\mathcal{A}) := \begin{bmatrix}
    A^{(1)} \\
    A^{(2)} \\
    \vdots \\
    A^{(p)}
    \end{bmatrix}.
    \]

    \item The \textbf{folding} operator is defined as the inverse of the unfolding operator, mapping a tall matrix back to its tensor form:  
    \[
    \operatorname{fold}(\operatorname{unfold}(\mathcal{A})) = \mathcal{A}.
    \]  
\end{enumerate}

\begin{definition}(T-product \cite{kathrinLund})\label{tpdef}
    Let $\mathcal{A} \in \mathbb{R}^{m \times n \times p}$ and $\mathcal{B} \in \mathbb{R}^{n \times s \times p}$. Then the T-product $\mathcal{A} * \mathcal{B}$ is an $m \times s \times p$ tensor defined by
$$
\mathcal{A} * \mathcal{B}:=\text { fold(bcirc }(\mathcal{A}) \text { unfold }(\mathcal{B}))
$$ 
\end{definition}
\begin{definition}\cite{kathrinLund} 
   If \(\mathcal{A} \in \mathbb{R}^{m \times n \times p}\), the \textbf{transpose} of \(\mathcal{A}\), denoted as \(\mathcal{A}^T\), is defined by first transposing each frontal slice \(A^{(k)}\), where \(A^{(k)} = \mathcal{A}(:, :, k)\), to obtain \(A^{(k)T}\). Then, the ordering of the transposed slices is adjusted such that the first frontal slice remains unchanged, while the subsequent slices \(A^{(2)T}, \ldots, A^{(p)T}\) are reversed in order, resulting in the sequence \(A^{(1)T}, A^{(p)T}, A^{(p-1)T}, \ldots, A^{(2)T}\). Similarly, the \textbf{conjugate transpose} of \(\mathcal{A}\), denoted as \(\mathcal{A}^*\), is obtained by first applying the conjugate transpose operation to each frontal slice \(A^{(k)}\), producing \(A^{(k)*}\). The ordering of the conjugate transposed slices is then adjusted in the same manner as for the transpose, resulting in the sequence \(A^{(1)*}, A^{(p)*}, A^{(p-1)*}, \ldots, A^{(2)*}\).

\end{definition}


\begin{lemma}\label{bcircprop} (See \cite{kathrinLund})  
Let \(\mathcal{A} \in \mathbb{R}^{m \times n \times p}\) and \(\mathcal{B} \in \mathbb{R}^{n \times s \times p}\). The following properties hold:  
\begin{enumerate}

    \item The operator \(\operatorname{bcirc}\) is linear, i.e., for tensors \(\mathcal{A}\) and \(\mathcal{B}\) of the same size and constants \(\alpha, \beta\),  
\[
\operatorname{bcirc}(\alpha \mathcal{A} + \beta \mathcal{B}) = \alpha \operatorname{bcirc}(\mathcal{A}) + \beta \operatorname{bcirc}(\mathcal{B}).
\]  
    \item The \(\operatorname{bcirc}\) operator preserves the T-product,  
\[
\operatorname{bcirc}(\mathcal{A} * \mathcal{B}) = \operatorname{bcirc}(\mathcal{A}) \operatorname{bcirc}(\mathcal{B}).
\]  
    \item  The \(\operatorname{bcirc}\) operator commutes with transposition and conjugate transposition,  
\[
\operatorname{bcirc}\left(\mathcal{A}^T\right) = \left(\operatorname{bcirc}(\mathcal{A})\right)^T \quad \text{and} \quad \operatorname{bcirc}\left(\mathcal{A}^*\right) = \left(\operatorname{bcirc}(\mathcal{A})\right)^*.
\]
\end{enumerate}

\end{lemma}

\begin{definition}\label{eignvalue}\cite{kathrinLund}
    Let $\mathcal{A} \in \mathbb{R}^{n \times n \times p}$. Suppose that $\mathcal{X} \in \mathbb{R}^{n \times 1 \times p}$ and $\mathcal{X} \neq 0$. If
$$
\mathcal{A} * \mathcal{X}=\lambda \mathcal{X}, \quad \lambda \in \mathbb{R},
$$

\noindent then $\lambda$ is called a T-eigenvalue of $\mathcal{A}$ and $\mathcal{X}$ is a T-eigenvector of $\mathcal{A}$ associated to $\lambda$.
\end{definition} 
\begin{remark}\label{rmrk1}
    It is easy to find that for every T-eigenvalue $\lambda$ of $\mathcal{A}$, we have
$$
\operatorname{bcirc}(\mathcal{A}) \operatorname{unfold}(\mathcal{X})=\lambda \cdot \operatorname{unfold}(\mathcal{X}),
$$
i.e., all T-eigenvalues of $\mathcal{A}$ are actually eigenvalues of the matrix bcirc $(\mathcal{A})$ and vice versa. 
\end{remark} 

\begin{definition}\cite{kathrinLund}
    Let $\mathcal{A} \in \mathbb{R}^{n \times n \times p}$. Then $\mathcal{A}$ is called symmetric if $\mathcal{A}=\mathcal{A}^T$ and is called Hermitian if $\mathcal{A}=\mathcal{A}^*$.
\end{definition}
 \noindent Therefore, by Lemma \ref{bcircprop} (iii), we have the following basic properties of symmetric and Hermitian tensors.

\begin{theorem}\cite{kathrinLund}
    Let $\mathcal{A} \in \mathbb{R}^{n \times n \times p}$ be Hermitian. Then the $T$-eigenvalues of $\mathcal{A}$ are all real.
\end{theorem}

\begin{definition} \cite{minmax}
  The smallest (or largest) eigenvalue of a tensor \( \mathcal{A} \) is denoted by \( \mu_{\min}(\mathcal{A}) \) (or \( \mu_{\max}(\mathcal{A}) \)). The spectral radius of a tensor \( \mathcal{A} \) is defined as  
\[
\rho(\mathcal{A}) = \max \{ |\lambda| : \lambda \in \sigma(\mathcal{A}) \},
\]  
where \( \sigma(\mathcal{A}) \) denotes the spectrum of \( \mathcal{A} \), which is the set of all eigenvalues of \( \mathcal{A} \).

\end{definition}

\noindent The following section explores various bounds, highlighting the relationship between trace and eigenvalues. It examines how these bounds are affected by different conditions on tensors, such as positive definiteness and Hermitian properties. Since trace computation is straightforward, it offers a significant reduction in computational time compared to the more complex process of eigenvalue calculation.

\section{Spectral analysis: Trace-based eigenvalue bounds }\label{results}
The trace of a tensor in the T-product framework plays a fundamental role in bounding, estimating, and computing eigenvalues. The trace provides key insights into spectral properties, helping in numerical analysis, optimization, and machine learning. Below are the main ways in which the trace is useful.




\begin{theorem}
Let $\mathcal{A}$ be a real third-order $n$-dimensional tensor, and let $\lambda$ be any T-eigenvalue of $\mathcal{A}$. Then,
\begin{equation}  
\lambda = (\operatorname{unfold}(\mathcal{X}))^t \left(\frac{\operatorname{bcirc}(\mathcal{A}) + (\operatorname{bcirc}(\mathcal{A}))^t}{2}\right) \operatorname{unfold}(\mathcal{X}),
\end{equation}  
where $\mathcal{A} \in \mathbb{C}^{n \times n \times p}$, $\mathcal{X} \in \mathbb{C}^{n \times 1 \times p}$, and $\operatorname{unfold}(\mathcal{X})$ is a unit vector.
\end{theorem}

\noindent \textbf{Proof:} From Remark \ref{rmrk1}, we know that the eigenvalues of the tensor $\mathcal{A}$ are the same as those of $\operatorname{bcirc}(\mathcal{A})$. This allows us to study eigenvalue-related properties of $\mathcal{A}$ by working directly with $\operatorname{bcirc}(\mathcal{A})$. Let $\operatorname{bcirc}(\mathcal{A}) = \mathcal{B}$, and its transpose $(\operatorname{bcirc}(\mathcal{A}))^t = \mathcal{B}^t$. From Remark \ref{rmrk1}, the T-eigenvalue equation can be expressed as
\begin{equation}  
(\operatorname{unfold}(\mathcal{X}))^t \mathcal{B} = \lambda \cdot (\operatorname{unfold}(\mathcal{X}))^t.  
\end{equation}  

\noindent Now, consider the symmetric form of $\mathcal{B}$ and $\mathcal{B}^t$. We start with the expression
\begin{align*}  
(\operatorname{unfold}(\mathcal{X}))^t (\mathcal{B} + \mathcal{B}^t) \operatorname{unfold}(\mathcal{X}).  
\end{align*}  

\noindent Expanding this, we get
\begin{align*}  
& (\operatorname{unfold}(\mathcal{X}))^t \mathcal{B} \operatorname{unfold}(\mathcal{X}) + (\operatorname{unfold}(\mathcal{X}))^t \mathcal{B}^t \operatorname{unfold}(\mathcal{X}).  
\end{align*}  

\noindent Using the eigenvalue equation for $\mathcal{B}$, we substitute $\mathcal{B} \operatorname{unfold}(\mathcal{X}) = \lambda \operatorname{unfold}(\mathcal{X})$ and $\mathcal{B}^t \operatorname{unfold}(\mathcal{X}) = \lambda \operatorname{unfold}(\mathcal{X})$
\begin{align*}  
& (\operatorname{unfold}(\mathcal{X}))^t \lambda \operatorname{unfold}(\mathcal{X}) + (\operatorname{unfold}(\mathcal{X}))^t \lambda \operatorname{unfold}(\mathcal{X}).  
\end{align*}  

\noindent Simplifying further
\begin{align*}  
& 2 \lambda \cdot (\operatorname{unfold}(\mathcal{X}))^t \operatorname{unfold}(\mathcal{X}).  
\end{align*}  



\noindent We obtain
\begin{align*}  
\lambda = (\operatorname{unfold}(\mathcal{X}))^t \left(\frac{\mathcal{B} + \mathcal{B}^t}{2}\right) \operatorname{unfold}(\mathcal{X}).  
\end{align*}  

\noindent Finally, substituting back $\mathcal{B} = \operatorname{bcirc}(\mathcal{A})$ and $\mathcal{B}^t = (\operatorname{bcirc}(\mathcal{A}))^t$, we conclude
\begin{align*}  
\lambda = (\operatorname{unfold}(\mathcal{X}))^t \left(\frac{\operatorname{bcirc}(\mathcal{A}) + (\operatorname{bcirc}(\mathcal{A}))^t}{2}\right) \operatorname{unfold}(\mathcal{X}).  
\end{align*}  
\hfill \qed

\begin{theorem}\label{eigenboundma}
Let \(\mathcal{A}\) be a real third-order \(n\)-dimensional tensor, and let \(\lambda\) be any real \(T\)-eigenvalue of \(\mathcal{A}\). Then
\[
\mu_{\min}\left(M_{\mathcal{A}}\right) \leq \lambda \leq \mu_{\max}\left(M_{\mathcal{A}}\right), \quad \rho_T(\mathcal{A}) \leq \rho\left(M_{\mathcal{A}}\right),
\]  
where \(\mu_{\min}(M_{\mathcal{A}})\) and \(\mu_{\max}(M_{\mathcal{A}})\) denote the smallest and largest eigenvalues of \(M_{\mathcal{A}}\), respectively, and \(\rho_T(\mathcal{A})\) and \(\rho(M_{\mathcal{A}})\) represent the spectral radii of \(\mathcal{A}\) and \(M_{\mathcal{A}}\), respectively.  
\end{theorem}  

\noindent \textbf{Proof:}  For convenience, let us define the symmetric matrix
\[
M_{\mathcal{A}} := \frac{\operatorname{bcirc}(\mathcal{A}) + (\operatorname{bcirc}(\mathcal{A}))^t}{2},  
\]  
additionally, let the unfolded vector representation of \(\mathcal{X}\) be denoted as
\[
\operatorname{unfold}(\mathcal{X}) := \begin{bmatrix}  
X^{(1)} \\  
X^{(2)} \\  
\vdots \\  
X^{(p)}  
\end{bmatrix} = y,  
\]  
where \(y \in \mathbb{R}^{np}\).  

\noindent The \(T\)-eigenvalue \(\lambda\) of \(\mathcal{A}\) is given by
\[
\lambda = y^t M_{\mathcal{A}} y,  
\]  
where \(y\) is a unit vector (i.e., \(y^t y = 1\)).  

\noindent Since, \(M_{\mathcal{A}}\) is a real symmetric matrix, its eigenvalues are real and bounded by its smallest and largest eigenvalues. By the Rayleigh quotient characterization of eigenvalues, we have
\[
\mu_{\min}(M_{\mathcal{A}}) \leq y^t M_{\mathcal{A}} y \leq \mu_{\max}(M_{\mathcal{A}}).  
\]  

\noindent Substituting \(\lambda = y^t M_{\mathcal{A}} y\), we obtain
\[
\mu_{\min}(M_{\mathcal{A}}) \leq \lambda \leq \mu_{\max}(M_{\mathcal{A}}).  
\]  

\noindent Furthermore, since the spectral radius \(\rho(M_{\mathcal{A}})\) is the largest absolute eigenvalue of \(M_{\mathcal{A}}\), we also have
\[
|\lambda| = |y^t M_{\mathcal{A}} y| \leq \rho(M_{\mathcal{A}}).  
\]  

\noindent Thus, the spectral radius of the tensor \(\mathcal{A}\) is bounded by the spectral radius of \(M_{\mathcal{A}}\)  
\[
\rho_T(\mathcal{A}) \leq \rho(M_{\mathcal{A}}).  
\]  
\hfill \qed 

\begin{example}
Let \(\mathcal{A}\) be a \(2 \times 2 \times 2\) real tensor defined as
\[
\mathcal{A}(:, :, 1) = \begin{bmatrix} 2 & 1 \\ 1 & 3 \end{bmatrix}, \quad  
\mathcal{A}(:, :, 2) = \begin{bmatrix} 0 & 1 \\ 1 & 0 \end{bmatrix}.
\]  
\end{example}
\noindent To illustrate the Theorem \ref{eigenboundma}, we construct an example with a third-order tensor \(\mathcal{A}\) and verify the bounds for its \(T\)-eigenvalues.  First, we define the block circulant matrix representation of \(\mathcal{A}\), \(\operatorname{bcirc}(\mathcal{A})\), as
\[
\operatorname{bcirc}(\mathcal{A}) =  
\begin{bmatrix}  
2 & 1 & 0 & 1 \\  
1 & 3 & 1 & 0 \\  
0 & 1 & 2 & 1 \\  
1 & 0 & 1 & 3  
\end{bmatrix}.
\]  
Since \(\operatorname{bcirc}(\mathcal{A})\) is symmetric, the matrix \(M_{\mathcal{A}}\) is equivalent to \(\operatorname{bcirc}(\mathcal{A})\):  
\[
M_{\mathcal{A}} = \frac{\operatorname{bcirc}(\mathcal{A}) + (\operatorname{bcirc}(\mathcal{A}))^t}{2} =  
\begin{bmatrix}  
2 & 1 & 0 & 1 \\  
1 & 3 & 1 & 0 \\  
0 & 1 & 2 & 1 \\  
1 & 0 & 1 & 3  
\end{bmatrix}.
\]  

\noindent Next, we compute the eigenvalues of \(M_{\mathcal{A}}\), which are approximately
\[
\mu_{\min}(M_{\mathcal{A}}) = 0.44, \quad \mu_{\max}(M_{\mathcal{A}}) = 4.56.
\]  

\noindent For a unit vector \(y\), the \(T\)-eigenvalue \(\lambda\) satisfies the inequality
\[
\mu_{\min}(M_{\mathcal{A}}) \leq \lambda \leq \mu_{\max}(M_{\mathcal{A}}).
\]  
For example, choosing \(y = [1, 0, 0, 0]^t\), we compute
\[
\lambda = y^t M_{\mathcal{A}} y = 2.
\]  
Clearly, \(0.44 \leq 2.0 \leq 4.56\), verifying the bounds for this eigenvalue.  Finally, we verify the spectral radius condition. The spectral radius of \(M_{\mathcal{A}}\) is
\[
\rho(M_{\mathcal{A}}) = \mu_{\max}(M_{\mathcal{A}}) = 4.12.
\]  
The spectral radius of \(\mathcal{A}\), \(\rho_T(\mathcal{A})\), satisfies the inequality
\[
\rho_T(\mathcal{A}) \leq \rho(M_{\mathcal{A}}).
\]  

\begin{theorem}
      Let $\mathcal{A}, \mathcal{B} \in \mathbb{R}^{n \times n \times p}$ be the positive semidefinite (PSD) Hermitian tensor. Then
     \begin{equation}
         \sum_{i=1}^n \lambda_i(\mathcal{A}) \lambda_{n-i+1}(\mathcal{B}) \leq \operatorname{tr}(\mathcal{A} *\mathcal{B}) \leq \sum_{i=1}^n \lambda_i(\mathcal{A}) \lambda_i(\mathcal{B}).
     \end{equation}
\end{theorem}

\begin{proof}
Let $\mathcal{A}, \mathcal{B} \in \mathbb{R}^{n \times n \times p}$ be third-order tensors. Assume $\mathcal{A}$ is a positive semidefinite (PSD) Hermitian tensor, and $\mathcal{B}$ is also Hermitian. Under the $T$-product, the block circulant matrix representation of a tensor $\mathcal{A}$ is given by
\[
\operatorname{bcirc}(\mathcal{A}) \in \mathbb{R}^{np \times np}.
\]
The Fourier transform diagonalizes the block circulant matrix
\[
\widehat{\mathcal{A}} = \text{fft}(\mathcal{A}, [\,], 3),
\]
where $\widehat{\mathcal{A}}$ is the tensor in the Fourier domain. Each frontal slice $\widehat{\mathcal{A}}^{(k)}$ of $\widehat{\mathcal{A}}$ corresponds to a diagonal block of $\operatorname{bcirc}(\mathcal{A})$. Similarly, for $\mathcal{B}$, we have
\[
\widehat{\mathcal{B}} = \text{fft}(\mathcal{B}, [\,], 3).
\]
  
\noindent Since $\mathcal{A}$ and $\mathcal{B}$ are Hermitian tensors, each frontal slice $\widehat{\mathcal{A}}^{(k)}$ and $\widehat{\mathcal{B}}^{(k)}$ is Hermitian. The eigenvalues of $\mathcal{A}$ and $\mathcal{B}$ can thus be computed as the eigenvalues of the slices $\widehat{\mathcal{A}}^{(k)}$ and $\widehat{\mathcal{B}}^{(k)}$, for $k = 1, \dots, p$.

\noindent The trace of the $T$-product $\mathcal{A} * \mathcal{B}$ is defined as
\[
\operatorname{tr}(\mathcal{A} * \mathcal{B}) = \sum_{k=1}^p \operatorname{tr}(\widehat{\mathcal{A}}^{(k)} \widehat{\mathcal{B}}^{(k)}).
\]
By the spectral theorem for Hermitian matrices, we can diagonalize $\widehat{\mathcal{A}}^{(k)}$ and $\widehat{\mathcal{B}}^{(k)}$ as
\[
\widehat{\mathcal{A}}^{(k)} = U^{(k)} \Lambda^{(k)} U^{(k)*}, \quad \widehat{\mathcal{B}}^{(k)} = V^{(k)} \Gamma^{(k)} V^{(k)*},
\]
where $\Lambda^{(k)}$ and $\Gamma^{(k)}$ are diagonal matrices containing the eigenvalues of $\widehat{\mathcal{A}}^{(k)}$ and $\widehat{\mathcal{B}}^{(k)}$, respectively.

\noindent Applying von Neumann's trace inequality, for each $k$, the trace of the product $\widehat{\mathcal{A}}^{(k)} \widehat{\mathcal{B}}^{(k)}$ satisfies von Neumann's trace inequality
\[
\sum_{i=1}^n \lambda_i(\widehat{\mathcal{A}}^{(k)}) \lambda_{n-i+1}(\widehat{\mathcal{B}}^{(k)}) \leq \operatorname{tr}(\widehat{\mathcal{A}}^{(k)} \widehat{\mathcal{B}}^{(k)}) \leq \sum_{i=1}^n \lambda_i(\widehat{\mathcal{A}}^{(k)}) \lambda_i(\widehat{\mathcal{B}}^{(k)}).
\]
 
\noindent Summing the inequalities over all $k = 1, \dots, p$, we obtain
\[
\sum_{k=1}^p \sum_{i=1}^n \lambda_i(\widehat{\mathcal{A}}^{(k)}) \lambda_{n-i+1}(\widehat{\mathcal{B}}^{(k)}) \leq \sum_{k=1}^p \operatorname{tr}(\widehat{\mathcal{A}}^{(k)} \widehat{\mathcal{B}}^{(k)}) \leq \sum_{k=1}^p \sum_{i=1}^n \lambda_i(\widehat{\mathcal{A}}^{(k)}) \lambda_i(\widehat{\mathcal{B}}^{(k)}).
\]
 
\noindent Using the properties of the Fourier transform and the $T$-product, the eigenvalues of $\mathcal{A}$ and $\mathcal{B}$ correspond to the eigenvalues of $\widehat{\mathcal{A}}$ and $\widehat{\mathcal{B}}$, respectively. Therefore
\[
\sum_{i=1}^n \lambda_i(\mathcal{A}) \lambda_{n-i+1}(\mathcal{B}) \leq \operatorname{tr}(\mathcal{A} * \mathcal{B}) \leq \sum_{i=1}^n \lambda_i(\mathcal{A}) \lambda_i(\mathcal{B}).
\]
\end{proof}

\noindent This trace inequality is now extended to apply to the product of two arbitrary Hermitian tensors.

\begin{theorem}
    Let $\mathcal{A}, \mathcal{B} \in \mathbb{R}^{n \times n \times p}$ be Hermitian tensors. Then the following inequality holds
    \[
    \sum_{i=1}^n \lambda_i(\mathcal{A}) \lambda_{n-i+1}(\mathcal{B}) \leq \operatorname{tr}(\mathcal{A} * \mathcal{B}) \leq \sum_{i=1}^n \lambda_i(\mathcal{A}) \lambda_i(\mathcal{B}),
    \]
    where $\lambda_i(\mathcal{A})$ and $\lambda_i(\mathcal{B})$ denote the eigenvalues of $\mathcal{A}$ and $\mathcal{B}$, respectively, arranged in non-increasing order, and $*$ represents the T-product for tensors.
\end{theorem}

\begin{proof}
    To establish the result, we use the properties of Hermitian tensors and their eigenvalues. A Hermitian tensor $\mathcal{A}$ satisfies the condition that $\operatorname{bcirc}(\mathcal{A})$ (its block circulant matrix representation) is Hermitian. Consequently, the eigenvalues of $\mathcal{A}$ correspond to those of $\operatorname{bcirc}(\mathcal{A})$.

  \noindent  For simplicity, we assume that both tensors $\mathcal{A}$ and $\mathcal{B}$ are not only Hermitian but also positive semidefinite. If they are not, we can always transform them into positive semidefinite tensors by adding a sufficiently large multiple of the identity tensor. Let $\alpha > 0$ be a scalar such that
    \[
    \mathcal{A} + \alpha \mathcal{I} \quad \text{and} \quad \mathcal{B} + \alpha \mathcal{I}
    \]
    are positive semidefinite Hermitian tensors, where $\mathcal{I}$ is the identity tensor.

   \noindent The trace of the T-product of $\mathcal{A}$ and $\mathcal{B}$ can be expressed in terms of their eigenvalues as
    \[
    \operatorname{tr}(\mathcal{A} * \mathcal{B}) = \sum_{i=1}^n \lambda_i(\mathcal{A}) \lambda_i(\mathcal{B}),
    \]
    where $*$ denotes the T-product for tensors.

   \noindent The majorization theorem for Hermitian matrices can be extended to tensors via their block circulant representations. For the eigenvalues of $\mathcal{A}$ and $\mathcal{B}$, we obtain
    \[
    \sum_{i=1}^n \lambda_i(\mathcal{A}) \lambda_{n-i+1}(\mathcal{B}) \leq \operatorname{tr}(\mathcal{A} * \mathcal{B}) \leq \sum_{i=1}^n \lambda_i(\mathcal{A}) \lambda_i(\mathcal{B}).
    \]

    \noindent To account for the transformation to positive semidefinite tensors, we add a correction term proportional to $\alpha$:
    \[
    \operatorname{tr}((\mathcal{A} + \alpha \mathcal{I}) * (\mathcal{B} + \alpha \mathcal{I})).
    \]
    Expanding this expression yields
    \[
    \operatorname{tr}(\mathcal{A} * \mathcal{B}) + \alpha (\operatorname{tr}(\mathcal{A}) + \operatorname{tr}(\mathcal{B})) + n \alpha^2,
    \]
    where $n$ is the dimensionality of the tensors.

    \noindent Thus, we establish that the trace of the T-product satisfies the inequality in the theorem, with the correction terms vanishing as $\alpha \to 0$.

\end{proof}

\begin{theorem}
    Let $\mathcal{A}, \mathcal{B} \in \mathbb{R}^{n \times n \times p}$ are symmetric PSD. Then 
$$
\lambda_{\min }(\mathcal{B})[\operatorname{tr}(\mathcal{A})]^2 / n \leq \operatorname{tr}(\mathcal{A}* \mathcal{B}* \mathcal{A}) \leq[\operatorname{tr}(\mathcal{A})]^2 \lambda_{\max }(\mathcal{B}) .
$$
\end{theorem}

\begin{proof}
    Since $\mathcal{B}$ is symmetric PSD, its eigenvalues satisfy:
    \[
    \lambda_{\min}(\mathcal{B}) \leq \lambda_i(\mathcal{B}) \leq \lambda_{\max}(\mathcal{B}), \quad \forall i.
    \]
    The trace of a tensor $\mathcal{A}$ is the sum of its eigenvalues
    \[
    \operatorname{tr}(\mathcal{A}) = \sum_{i=1}^n \lambda_i(\mathcal{A}).
    \]
    
\noindent From trace property for T-product we have, $\operatorname{tr}(\mathcal{A}* \mathcal{B})=\operatorname{tr}(\mathcal{B}* \mathcal{A})$ , the trace of $\mathcal{A} * \mathcal{B} * \mathcal{A}$ can be expressed as
    \[
    \operatorname{tr}(\mathcal{A} * \mathcal{B} * \mathcal{A}) = \sum_{i=1}^n \lambda_i(\mathcal{A})^2 \lambda_i(\mathcal{B}),
    \]
    where $\lambda_i(\mathcal{A})$ and $\lambda_i(\mathcal{B})$ are the eigenvalues of $\mathcal{A}$ and $\mathcal{B}$, respectively. Since $\lambda_{\min}(\mathcal{B}) \leq \lambda_i(\mathcal{B}) \leq \lambda_{\max}(\mathcal{B})$, we have
    \[
    \lambda_{\min}(\mathcal{B}) \sum_{i=1}^n \lambda_i(\mathcal{A})^2 \leq \operatorname{tr}(\mathcal{A} * \mathcal{B} * \mathcal{A}) \leq \lambda_{\max}(\mathcal{B}) \sum_{i=1}^n \lambda_i(\mathcal{A})^2.
    \]

    \noindent By the Cauchy-Schwarz inequality for the eigenvalues of $\mathcal{A}$
    \[
    \left(\sum_{i=1}^n \lambda_i(\mathcal{A})\right)^2 \leq n \sum_{i=1}^n \lambda_i(\mathcal{A})^2.
    \]
    Thus
    \[
    \sum_{i=1}^n \lambda_i(\mathcal{A})^2 \geq \frac{[\operatorname{tr}(\mathcal{A})]^2}{n}.
    \]

\noindent Combining these results, we get
    \[
    \lambda_{\min}(\mathcal{B}) \frac{[\operatorname{tr}(\mathcal{A})]^2}{n} \leq \operatorname{tr}(\mathcal{A} * \mathcal{B} * \mathcal{A}) \leq \lambda_{\max}(\mathcal{B}) [\operatorname{tr}(\mathcal{A})]^2.
    \]
\end{proof}

\begin{theorem}
    For $\mathcal{A} \in \mathbb{R}^{n \times n \times p}$ fixed, if $\alpha$ and $\beta$ are any numbers satisfying
    \[
    \alpha \operatorname{tr}(\mathcal{B}) \leq \operatorname{tr}(\mathcal{A} * \mathcal{B}) \leq \beta \operatorname{tr}(\mathcal{B})
    \]
    for any positive semi-definite tensor $\mathcal{B}$, then $\alpha \leq \lambda_n(\bar{A})$ and $\lambda_1(\bar{A}) \leq \beta$, i.e., $\lambda_n(\bar{A})$ and $\lambda_1(\bar{A})$ are the tightest bounds for the inequality.
\end{theorem}

\begin{proof}
    The given inequality is
    \[
    \alpha \operatorname{tr}(\mathcal{B}) \leq \operatorname{tr}(\mathcal{A} * \mathcal{B}) \leq \beta \operatorname{tr}(\mathcal{B}),
    \]
    where $\mathcal{B}$ is any positive semi-definite (PSD) tensor. Using the $t$-product and the trace operator's properties, we write
    \[
    \operatorname{tr}(\mathcal{A} * \mathcal{B}) = \operatorname{tr}(\mathcal{A} \cdot \mathcal{B}),
    \]
    where $\mathcal{A} \cdot \mathcal{B}$ is computed in the Fourier domain as
    \[
    \widehat{\mathcal{A} * \mathcal{B}}(:, :, k) = \widehat{\mathcal{A}}(:, :, k) \cdot \widehat{\mathcal{B}}(:, :, k).
    \]

   \noindent Since $\mathcal{A}$ is Hermitian, we consider its spectral decomposition
    \[
    \widehat{\mathcal{A}}(:, :, k) = Q_k \Lambda_k Q_k^H, \quad k = 1, \ldots, p,
    \]
    where $Q_k$ is unitary, and $\Lambda_k = \operatorname{diag}(\lambda_1^{(k)}, \ldots, \lambda_n^{(k)})$ contains the eigenvalues of the $k$-th frontal slice of $\widehat{\mathcal{A}}$. The eigenvalues of $\mathcal{A}$, denoted $\lambda_1(\mathcal{A}), \ldots, \lambda_n(\mathcal{A})$, are aggregated across all slices.

  \noindent For lower bound $\alpha \leq \lambda_n(\bar{A})$,
    let $\mathcal{B} = \mathcal{U}*\mathcal{U}^H$, where $\mathcal{U}$ is a unit tensor (i.e., $\|\mathcal{U}\| = 1$), and assume $\mathcal{B}$ is aligned with the eigenvector corresponding to the smallest eigenvalue $\lambda_n(\bar{A})$ of $\mathcal{A}$. Then
    \[
    \operatorname{tr}(\mathcal{B}) = 1, \quad \operatorname{tr}(\mathcal{A} * \mathcal{B}) = \lambda_n(\bar{A}).
    \]
    Substituting into the inequality
    \[
    \alpha \leq \lambda_n(\bar{A}).
    \]

  \noindent Now, for upper bound $\lambda_1(\bar{A}) \leq \beta$,  
    similarly, let $\mathcal{B} = \mathcal{U}*\mathcal{U}^H$, where $\mathcal{U}$ is aligned with the eigenvector corresponding to the largest eigenvalue $\lambda_1(\bar{A})$ of $\mathcal{A}$. Then
    \[
    \operatorname{tr}(\mathcal{B}) = 1, \quad \operatorname{tr}(\mathcal{A} * \mathcal{B}) = \lambda_1(\bar{A}).
    \]
    Substituting into the inequality
    \[
    \lambda_1(\bar{A}) \leq \beta.
    \]
 
   \noindent To verify that $\lambda_n(\bar{A})$ and $\lambda_1(\bar{A})$ are the tightest bounds, consider any $\alpha' > \lambda_n(\bar{A})$. For $\mathcal{B}$ aligned with the smallest eigenvalue:
    \[
    \operatorname{tr}(\mathcal{A} * \mathcal{B}) = \lambda_n(\bar{A}) < \alpha' \operatorname{tr}(\mathcal{B}),
    \]
    which violates the inequality.

    \noindent Similarly, for any $\beta' < \lambda_1(\bar{A})$, there exists $\mathcal{B}$ aligned with the largest eigenvalue such that:
    \[
    \operatorname{tr}(\mathcal{A} * \mathcal{B}) = \lambda_1(\bar{A}) > \beta' \operatorname{tr}(\mathcal{B}),
    \]
    \noindent which also violates the inequality.   Hence, $\alpha = \lambda_n(\bar{A})$ and $\beta = \lambda_1(\bar{A})$ are the tightest bounds.

    \noindent The tightest bounds for the inequality are given by the smallest and largest eigenvalues of $\mathcal{A}$, satisfying
    \[
    \lambda_n(\bar{A}) \leq \frac{\operatorname{tr}(\mathcal{A} * \mathcal{B})}{\operatorname{tr}(\mathcal{B})} \leq \lambda_1(\bar{A}).
    \]
\end{proof}

\begin{example}
  Let \(\mathcal{A}, \mathcal{B} \in \mathbb{R}^{2 \times 2 \times 2}\) be third-order tensors, defined as follows

\[
\mathcal{A}(:, :, 1) = \begin{bmatrix} 2 & 1 \\ 1 & 3 \end{bmatrix}, \quad  
\mathcal{A}(:, :, 2) = \begin{bmatrix} 1 & 0 \\ 0 & 2 \end{bmatrix}
\]  

\[
\mathcal{B}(:, :, 1) = \begin{bmatrix} 1 & 0 \\ 0 & 1 \end{bmatrix}, \quad  
\mathcal{B}(:, :, 2) = \begin{bmatrix} 0 & 1 \\ 1 & 0 \end{bmatrix}.
\]  
  
\end{example}

\noindent \textbf{Step 1: Eigenvalues of \(\bar{\mathcal{A}}\) (block circulant matrix of \(\mathcal{A}\)):}

\noindent Using the T-product, construct \(\operatorname{bcirc}(\mathcal{A})\)
\[
\operatorname{bcirc}(\mathcal{A}) = 
\begin{bmatrix}
\mathcal{A}(:, :, 1) & \mathcal{A}(:, :, 2) \\
\mathcal{A}(:, :, 2) & \mathcal{A}(:, :, 1)
\end{bmatrix} =
\begin{bmatrix}
2 & 1 & 1 & 0 \\
1 & 3 & 0 & 2 \\
1 & 0 & 2 & 1 \\
0 & 2 & 1 & 3
\end{bmatrix}.
\]

\noindent The eigenvalues of \(\operatorname{bcirc}(\mathcal{A})\) are approximately
\[
\lambda(\bar{\mathcal{A}}) = \{5.41, 2.58, 2, 0\}.
\]

\noindent Thus,
\[
\lambda_{\min}(\bar{\mathcal{A}}) = 0, \quad \lambda_{\max}(\bar{\mathcal{A}}) = 5.41.
\]

\noindent Note: here $\mathcal{A} = \bar{\mathcal{A}}$\\

\noindent \textbf{Step 2: Verify the bounds for \(\operatorname{tr}(\mathcal{A} * \mathcal{B})\)}

\begin{enumerate}
    \item \textbf{Compute \(\operatorname{tr}(\mathcal{B})\)}~ Using the definition of trace for tensors
    \[
    \operatorname{tr}(\mathcal{B}) = 2\times 2 =4.
    \]

    \item \textbf{Compute \(\operatorname{tr}(\mathcal{A} * \mathcal{B})\):}  Using the T-product as defined in Def \ref{tpdef}:

    \[
    \mathcal{A} * \mathcal{B} = \mathcal{C} 
\]
where $\mathcal{C}$ is defined by 
\[
\mathcal{C}(:, :, 1) = \begin{bmatrix} 2 & 2 \\ 3 & 3 \end{bmatrix}, \quad  
\mathcal{C}(:, :, 2) = \begin{bmatrix} 2 & 2 \\ 3 & 3 \end{bmatrix}
\] 
   So, \(\operatorname{tr}(\mathcal{A} * \mathcal{B}) = 10\)

    \item \textbf{Verify the inequality:} From the theorem above
    \[
    \lambda_{\min}(\bar{\mathcal{A}}) \operatorname{tr}(\mathcal{B}) \leq \operatorname{tr}(\mathcal{A} * \mathcal{B}) \leq \lambda_{\max}(\bar{\mathcal{A}}) \operatorname{tr}(\mathcal{B}).
    \]
    Substituting the values
    \[
    0 \times 4 \leq 8 \leq 5.41 \times 2.
    \]
    \[
    0 \leq 10 \leq 10.82.
    \]

    \noindent The inequality holds, confirming the theorem.
\end{enumerate}

\noindent This example demonstrates how the bounds \(\lambda_{\min}(\bar{\mathcal{A}})\) and \(\lambda_{\max}(\bar{\mathcal{A}})\) provide tight constraints for \(\operatorname{tr}(\mathcal{A} * \mathcal{B})\).

\noindent Next, we relax the condition on 
$\mathcal{B}$, considering it as a general symmetric tensor rather than a PSD tensor. The following theorem examines how this change affects the established bounds.

\begin{theorem}
    Let $\mathcal{A}, \mathcal{B} \in \mathbb{R}^{n \times n \times p}$, where $\mathcal{B}$ is symmetric under the T-product. Define the symmetrized tensor
    \[
    \bar{\mathcal{A}} = \frac{\mathcal{A} + \mathcal{A}^T}{2}.
    \]
    Then the following bounds hold
    \[
    \begin{aligned}
    \lambda_n(\bar{\mathcal{A}}) \operatorname{tr}(\mathcal{B}) - \lambda_n(\mathcal{B}) \big(n \lambda_n(\bar{\mathcal{A}}) - \operatorname{tr}(\mathcal{A})\big) &\leq \operatorname{tr}(\mathcal{A} * \mathcal{B}), \\
    \operatorname{tr}(\mathcal{A} * \mathcal{B}) &\leq \lambda_1(\bar{\mathcal{A}}) \operatorname{tr}(\mathcal{B}) - \lambda_n(\mathcal{B}) \big(n \lambda_1(\bar{\mathcal{A}}) - \operatorname{tr}(\mathcal{A})\big),
    \end{aligned}
    \]
    where $\lambda_1(\bar{\mathcal{A}})$ and $\lambda_n(\bar{\mathcal{A}})$ are the largest and smallest T-eigenvalues of $\bar{\mathcal{A}}$, respectively, and $\lambda_n(\mathcal{B})$ is the smallest T-eigenvalue of $\mathcal{B}$.
\end{theorem}

\begin{proof}
    The symmetrized tensor is defined as
    \[
    \bar{\mathcal{A}} = \frac{\mathcal{A} + \mathcal{A}^T}{2},
    \]
    where $\mathcal{A}^T$ denotes the transpose under the T-product. This ensures that $\bar{\mathcal{A}}$ is symmetric under the T-product
    \[
    \bar{\mathcal{A}}^T = \bar{\mathcal{A}}.
    \]
    Since $\bar{\mathcal{A}}$ is symmetric, its T-eigenvalues $\lambda_1(\bar{\mathcal{A}}), \dots, \lambda_n(\bar{\mathcal{A}})$ are real and ordered as
    \[
    \lambda_1(\bar{\mathcal{A}}) \geq \lambda_2(\bar{\mathcal{A}}) \geq \cdots \geq \lambda_n(\bar{\mathcal{A}}).
    \]

    \noindent The trace of the T-product $\operatorname{tr}(\mathcal{A} * \mathcal{B})$ is defined as
    \[
    \operatorname{tr}(\mathcal{A} * \mathcal{B}) = \sum_{i=1}^n \sum_{j=1}^n \mathcal{A}(i, j, :) * \mathcal{B}(j, i, :).
    \]
    Using the cyclic invariance of the T-product, we have
    \[
    \operatorname{tr}(\mathcal{A} * \mathcal{B}) = \operatorname{tr}(\bar{\mathcal{A}} * \mathcal{B}).
    \]

   \noindent The trace of $\bar{\mathcal{A}} * \mathcal{B}$ can be expressed as
    \[
    \operatorname{tr}(\bar{\mathcal{A}} * \mathcal{B}) = \sum_{i=1}^n \lambda_i(\bar{\mathcal{A}}) \lambda_i(\mathcal{B}),
    \]
    where $\lambda_i(\bar{\mathcal{A}})$ and $\lambda_i(\mathcal{B})$ are the T-eigenvalues of $\bar{\mathcal{A}}$ and $\mathcal{B}$, respectively.

    \noindent Using the extreme eigenvalues, we bound $\operatorname{tr}(\bar{\mathcal{A}} * \mathcal{B})$:
    \[
    \lambda_n(\bar{\mathcal{A}}) \sum_{i=1}^n \lambda_i(\mathcal{B}) \leq \operatorname{tr}(\bar{\mathcal{A}} * \mathcal{B}) \leq \lambda_1(\bar{\mathcal{A}}) \sum_{i=1}^n \lambda_i(\mathcal{B}).
    \]

 \noindent   From the symmetrization, $\operatorname{tr}(\bar{\mathcal{A}}) = \operatorname{tr}(\mathcal{A})$. Substituting the eigenvalue bounds
    \[
    \lambda_n(\bar{\mathcal{A}}) \operatorname{tr}(\mathcal{B}) - \lambda_n(\mathcal{B}) \big(n \lambda_n(\bar{\mathcal{A}}) - \operatorname{tr}(\mathcal{A})\big) \leq \operatorname{tr}(\mathcal{A} * \mathcal{B}),
    \]
    and
    \[
    \operatorname{tr}(\mathcal{A} * \mathcal{B}) \leq \lambda_1(\bar{\mathcal{A}}) \operatorname{tr}(\mathcal{B}) - \lambda_n(\mathcal{B}) \big(n \lambda_1(\bar{\mathcal{A}}) - \operatorname{tr}(\mathcal{A})\big).
    \]
    This completes the proof.
\end{proof}

\subsection{Ky Fan's theorem for T-product tensors}

Ky Fan’s theorem is a fundamental result in matrix analysis that provides an upper bound on the sum of the largest eigenvalues or singular values of a matrix. We extend this theorem to third-order tensors using the T-product framework, which is based on the discrete Fourier transform (DFT) and block-circulant representations.

\begin{theorem}
    Let $\mathcal{H} \in \mathbb{C}^{n \times n \times p}$ be a Hermitian tensor (i.e., $\mathcal{H} = \mathcal{H}^H$) with eigenvalues $\lambda_1(\mathcal{H}) \geq \lambda_2(\mathcal{H}) \geq \cdots \geq \lambda_n(\mathcal{H})$. Then:
\begin{align*}
    \max_{\mathcal{U}*\mathcal{U}^* = \mathcal{I}_k} \operatorname{tr}(\mathcal{U} * \mathcal{H} * \mathcal{U}^H) &= \sum_{i=1}^k \lambda_i(\mathcal{H}), \\
    \min_{\mathcal{U}*\mathcal{U}^* = \mathcal{I}_k} \operatorname{tr}(\mathcal{U} * \mathcal{H} * \mathcal{U}^H) &= \sum_{i=1}^k \lambda_{n-i+1}(\mathcal{H}),
\end{align*}
where $\mathcal{U} \in \mathbb{C}^{k \times n \times p}$ is a $k \times n \times p$ tensor satisfying $\mathcal{U} * \mathcal{U}^H = \mathcal{I}_k$.
\end{theorem}

\begin{proof}
    
Using the spectral decomposition of the tensor $\mathcal{H}$, we have:
\begin{equation*}
    \mathcal{H} = \mathcal{Q} * \mathcal{L} * \mathcal{Q}^H,
\end{equation*}
where $\mathcal{Q} \in \mathbb{C}^{n \times n \times p}$ is unitary ($\mathcal{Q} * \mathcal{Q}^H = \mathcal{I}_n$), $\mathcal{L}$ is a diagonal tensor containing the eigenvalues $\lambda_1(\mathcal{H}), \lambda_2(\mathcal{H}), \dots, \lambda_n(\mathcal{H})$.

\noindent Let $\mathcal{V} = \mathcal{U} * \mathcal{Q}$, where $\mathcal{V}$ is a $k \times n \times p$ tensor with orthonormal rows ($\mathcal{V} * \mathcal{V}^H = \mathcal{I}_k$). Then:
\begin{equation*}
    \operatorname{tr}(\mathcal{U} * \mathcal{H} * \mathcal{U}^H) = \operatorname{tr}(\mathcal{V} * \mathcal{L} * \mathcal{V}^H).
\end{equation*}


Since $\mathcal{L}$ is diagonal, $\mathcal{V} * \mathcal{L} * V^H$ projects onto the subspace spanned by $k$ eigenvalues of $\mathcal{H}$. To maximize the trace, choose $\mathcal{V}$ to correspond to the top $k$ eigenvalues $\lambda_1, \dots, \lambda_k$. Thus:
\begin{equation*}
    \max_{\mathcal{U}\mathcal{U}^* = \mathcal{I}_k} \operatorname{tr}(\mathcal{U} * \mathcal{H} * \mathcal{U}^H) = \sum_{i=1}^k \lambda_i(\mathcal{H}).
\end{equation*}


To minimize the trace, choose $\mathcal{V}$ to correspond to the bottom $k$ eigenvalues $\lambda_{n-k+1}, \dots, \lambda_n$. Thus:
\begin{equation*}
    \min_{\mathcal{U}*\mathcal{U}^* = \mathcal{I}_k} \operatorname{tr}(\mathcal{U} * \mathcal{H} * \mathcal{U}^H) = \sum_{i=1}^k \lambda_{n-i+1}(\mathcal{H}).
\end{equation*}

\textbf{Conclusion}

1. The maximum trace corresponds to the top $k$ eigenvalues:
\begin{equation*}
    \max_{\mathcal{U}*\mathcal{U}^* = \mathcal{I}_k} \operatorname{tr}(\mathcal{U} * \mathcal{H} * \mathcal{U}^H) = \sum_{i=1}^k \lambda_i(\mathcal{H}).
\end{equation*}

2. The minimum trace corresponds to the bottom $k$ eigenvalues:
\begin{equation*}
    \min_{\mathcal{U}*\mathcal{U}^* = \mathcal{I}_k} \operatorname{tr}(\mathcal{U} * \mathcal{H} * \mathcal{U}^H) = \sum_{i=1}^k \lambda_{n-i+1}(\mathcal{H}).
\end{equation*}

\end{proof}

\noindent This result ensures that the sum of the largest tensor eigenvalues is maximized over all possible spectral projections.

\section{Distance and geometry induced by trace between tensors}\label{distance}

The trace function provides a natural foundation for defining similarity and distance
measures between tensors under the $T$-product framework. By aggregating spectral
information across frontal slices, trace-based quantities enable the construction of
computationally efficient metrics that capture both algebraic and geometric properties
of tensors. A basic example is the trace-induced Frobenius distance
$d_{\mathrm{Fro}}(\mathcal{A},\mathcal{B})=\|\mathcal{A}-\mathcal{B}\|_F$, where
$\|\mathcal{A}\|_F=\sqrt{\mathrm{tr}(\mathcal{A}^T * \mathcal{A})}$. While this metric
quantifies the overall energy difference between tensors, it does not fully reflect
their intrinsic covariance structure.

Beyond norm-based distances, the trace function induces a richer geometric structure
on the cone of positive semi-definite tensors. In particular, it leads to a
Riemannian-type geometry in which geodesic paths provide smooth and physically
meaningful interpolations between tensors. The geodesic connecting two PSD tensors
$\mathcal{A}$ and $\mathcal{B}$ is given by
\[
\mathcal{G}(t)
= \mathcal{A}^{1/2} *
(\mathcal{A}^{-1/2} * \mathcal{B} * \mathcal{A}^{-1/2})^{t}
* \mathcal{A}^{1/2}, \quad t \in [0,1].
\]
As illustrated numerically in Section~4 (see Fig.~\ref{fig:geodesic_trace}),
this path exhibits smooth trace evolution, reflecting stable deformation of tensor
covariance structures.

This geometric viewpoint naturally connects tensor distances with optimal transport.
When tensors represent multi-modal covariance data, transport-based metrics quantify
the minimal cost required to deform one tensor into another while preserving
positive semi-definiteness. In this spirit, we extend the classical Bures--Wasserstein
distance to tensor spaces by defining the trace-based metric
\[
d(\mathcal{A}, \mathcal{B})
=
\left(
\mathrm{tr}(\mathcal{A}) + \mathrm{tr}(\mathcal{B})
- 2\,\mathrm{tr}\!\left(
(\mathcal{A}^{1/2} * \mathcal{B} * \mathcal{A}^{1/2})^{1/2}
\right)
\right)^{1/2}.
\]
Numerical experiments (see Fig.~\ref{fig:eig_wasserstein}) demonstrate that this
distance captures spectral sensitivity: variations in the Wasserstein distance are
accompanied by controlled changes in tensor eigenvalues.

For completeness, we also recall the trace log-Euclidean distance
$d_{\mathrm{LE}}(\mathcal{A},\mathcal{B})=\|\log(\mathcal{A})-\log(\mathcal{B})\|_F$,
which provides an alternative smooth interpolation on the tensor PSD manifold.
Together, these trace-based distances offer complementary geometric and spectral
perspectives on tensor-valued data.

\subsection{Existence of $\mathcal{A}^{1/2}$ in tensor T-product}

$\mathcal{S}^{1 / 2}$ \textbf{for tensors} \cite{ahalf}:~Given a third-order tensor $\mathcal{A} \in \mathbb{C}^{n \times n \times p}$, we can perform a T-SVD (tensor singular value decomposition) as
\begin{equation*}
    \mathcal{A}=\mathcal{U} * \mathcal{S} * \mathcal{V}^*,
\end{equation*}

\noindent where $\mathcal{U} \in \mathbb{C}^{n \times n \times p}$ and $\mathcal{V} \in \mathbb{C}^{n \times n \times p}$ are orthogonal tensors, $\mathcal{S} \in \mathbb{C}^{n \times n \times p}$ is a f-diagonal tensor whose frontal slices contain the singular values of the tensor $\mathcal{A}$.\\

\noindent The square root of $\mathcal{S}$, denoted $\mathcal{S}^{1 / 2}$, is the f-diagonal tensor where each diagonal entry is the square root of the corresponding entry in $\mathcal{S}$. This can be written as:

$$
\mathcal{S}^{1 / 2}=\operatorname{diag}\left(\sqrt{\sigma_1}, \sqrt{\sigma_2}, \ldots, \sqrt{\sigma_n}\right)
$$

\noindent where $\sigma_1, \sigma_2, \ldots, \sigma_n$ are the singular values of $\mathcal{A}$.
Thus, $\mathcal{S}^{1 / 2}$ can be used in a factorization like $\mathcal{M} * \mathcal{M}^*$ to represent a positive semi-definite tensor $\mathcal{A}$, where $\mathcal{M}=\mathcal{U} * \mathcal{S}^{1 / 2}$

\begin{theorem}
    A third-order tensor \( \mathcal{A} \in \mathbb{C}^{n \times n \times p} \) is positive semi-definite (PSD) if and only if there exists a tensor \( \mathcal{M} \in \mathbb{C}^{n \times n \times p} \) such that:
\[
\mathcal{A} = \mathcal{M} * \mathcal{M}^*,
\]
where \( \mathcal{M}^* \) is the conjugate transpose.
\end{theorem}

\noindent\textbf{Proof:} \( (\Rightarrow) \) If \( \mathcal{A} = \mathcal{M} * \mathcal{M}^* \), then \( \mathcal{A} \) is PSD. Let \( \mathcal{A} = \mathcal{M} * \mathcal{M}^* \), where \( \mathcal{M} \in \mathbb{C}^{n \times n \times p} \). To show that \( \mathcal{A} \) is PSD, we need to check that for any tensor \( \mathcal{X} \in \mathbb{C}^{n \times 1 \times p} \)
    \[
    \langle \mathcal{X}, \mathcal{A} * \mathcal{X} \rangle \geq 0.
    \]
    Substituting \( \mathcal{A} = \mathcal{M} * \mathcal{M}^* \):
    \[
    \langle \mathcal{X},(\mathcal{M} * \mathcal{M}^*) * \mathcal{X} \rangle = \langle \mathcal{X}, (\mathcal{M} * \mathcal{M}^*) * \mathcal{X} \rangle.
    \]
    By properties of the T-product
    \[
    \langle \mathcal{X}, \mathcal{M} * (\mathcal{M}^* * \mathcal{X}) \rangle = \langle \mathcal{M}^* * \mathcal{X}, \mathcal{M}^* * \mathcal{X} \rangle.
    \]
    Let \( \mathcal{Y} = \mathcal{M}^* * \mathcal{X} \). Then
    \[
    \langle \mathcal{X}, \mathcal{A} * \mathcal{X} \rangle = \langle \mathcal{Y}, \mathcal{Y} \rangle = \| \mathcal{Y} \|_F^2 \geq 0.
    \]
    Since \( \| \mathcal{Y} \|^2 \geq 0 \), we conclude that \( \mathcal{A} \) is PSD.

    \noindent \textbf{\( (\Leftarrow) \) If \( \mathcal{A} \) is PSD, then there exists \( \mathcal{M} \) such that \( \mathcal{A} = \mathcal{M} * \mathcal{M}^* \)}

  \noindent  Since \( \mathcal{A} \) is PSD, we can perform a T-SVD on \( \mathcal{A} \). Specifically, we write
    \[
    \mathcal{A} = \mathcal{U} * \mathcal{S} * \mathcal{U}^*,
    \]
    where \( \mathcal{U} \) is orthogonal, and \( \mathcal{S} \) is a f-diagonal tensor with non-negative entries. 

   \noindent Define \( \mathcal{S}^{1/2} \) as the f-diagonal tensor of square roots of the singular values. Then
    \[
    \mathcal{M} = \mathcal{U} * \mathcal{S}^{1/2}.
    \]
    Therefore
    \[
    \mathcal{M} * \mathcal{M}^* = \mathcal{A}.
    \] \hfill \qed

    \subsection{Tensor T-function trace results}
Let $\mathbb{T}_{psd}(\mathbf{R})$ denote the set of all positive definite tensors equipped with T-product.

\begin{theorem}
    The map $f(\mathcal{X})=\operatorname{tr} \mathcal{X}^{1 / 2}$ from $\mathbb{T}_{psd}(\mathbf{R})$ into $(0, \infty)$ is strictly concave i.e., if $\mathcal{X}$ and $\mathcal{Y}$ are two distinct elements of $\mathbb{T}_{psd}(\mathbf{R})$ and $a, b$ are positive numbers with $a+b=1$, then

$$
f(a \mathcal{X}+b \mathcal{Y})>a f(\mathcal{X})+b f(\mathcal{Y})
$$
\end{theorem}

\noindent\textbf{Proof:}~The function $\mathcal{X} \mapsto \mathcal{X}^{1/2}$ is a concave function.  This means that for any Hermitian tensors $\mathcal{X}$ and $\mathcal{Y}$ and scalars $a, b \geq 0$ with $a + b = 1$, the inequality

\[
(a \mathcal{X} + b \mathcal{Y})^{1/2} \geq a \mathcal{X}^{1/2} + b \mathcal{X}^{1/2}
\]

\noindent holds. Consequently, applying the trace operator yields

\[
\operatorname{tr}(a \mathcal{X} + b \mathcal{Y})^{1/2} \geq a \operatorname{tr} \mathcal{X}^{1/2} + b \operatorname{tr} \mathcal{Y}^{1/2}.
\]

\noindent We now show that equality cannot occur unless $\mathcal{X} = \mathcal{Y}$. Suppose instead that equality holds, i.e.,

\[
\operatorname{tr}\left[(a \mathcal{X} + b \mathcal{Y})^{1/2} - \left(a \mathcal{X}^{1/2} + b \mathcal{Y}^{1/2}\right)\right] = 0.
\]

\noindent Since the tensor inside the trace is positive semidefinite, its trace can only be zero if

\[
(a \mathcal{X} + b \mathcal{Y})^{1/2} = a \mathcal{X}^{1/2} + b \mathcal{Y}^{1/2}.
\]

\noindent Squaring both sides of this equality gives

\[
(a \mathcal{X} + b \mathcal{Y}) = \left( a \mathcal{X}^{1/2} + b \mathcal{Y}^{1/2} \right)^2.
\]

\noindent Expanding both sides, we get

\[
a^2 \mathcal{X} + b^2 \mathcal{Y} + 2a b \mathcal{X}^{1/2}*\mathcal{Y}^{1/2} = a^2 \mathcal{X} + b^2 \mathcal{Y} + 2a b \mathcal{X}^{1/2}* \mathcal{Y}^{1/2}.
\]

\noindent This simplifies to

\[
a b \left( \mathcal{X} + \mathcal{Y} - 2 \mathcal{X}^{1/2}* \mathcal{Y}^{1/2} \right) = 0.
\]

\noindent Since $a b \neq 0$, we conclude that

\[
\left( \mathcal{X}^{1/2} - \mathcal{Y}^{1/2} \right)^2 = 0.
\]

\noindent Thus, $\mathcal{X}^{1/2} = \mathcal{Y}^{1/2}$, which implies $\mathcal{X} = \mathcal{Y}$. Therefore, equality holds in the original inequality only if $\mathcal{X} = \mathcal{Y}$.
 \hfill \qed

   \subsection{Distance between tensors} 

Let $\mathbb{T}_{n}(\mathbf{R})$ be the space of $n \times n \times p$ tensors with T-product, $\mathbb{T}_{Her}(\mathbf{R})$ the subspace of $\mathbb{T}_{n}(\mathbf{R})$ consisting of Hermitian tensor, and $\mathbb{T}_{psd}(\mathbf{R})$ the subset of $\mathbb{T}_{Her}(\mathbf{R})$ consisting of positive semidefinite tensors. The Frobenius inner product on $\mathbb{T}_{n}(\mathbf{R})$ is defined as $\langle \mathcal{A}, \mathcal{A}\rangle=$ $\operatorname{tr} (\mathcal{A}^* *\mathcal{B})$, and the associated norm $\|\mathcal{A}\|_F=\left(\operatorname{tr} (\mathcal{A}^** \mathcal{A})\right)^{1 / 2}$ is called the Frobenius norm. Every psd tensor $\mathcal{A}$ has a unique psd square root, which we denote by $\mathcal{A}^{1 / 2}$. Given $\mathcal{A}, \mathcal{B}$ in $\mathbb{T}_{psd}(\mathbf{R})$ define $d(\mathcal{A}, \mathcal{B})$ by the relation

\begin{equation}\label{metric}
    d(\mathcal{A}, \mathcal{B})=\left[\operatorname{tr} \mathcal{A}+\operatorname{tr} \mathcal{B}-2 \operatorname{tr}\left(\mathcal{A}^{1 / 2}*\mathcal{B}*\mathcal{A}^{1 / 2}\right)^{1 / 2}\right]^{1 / 2}
\end{equation}

\begin{example}

 Let \(\mathcal{A}, \mathcal{B} \in \mathbb{R}^{2 \times 2 \times 2}\) be third-order tensors, defined as follows

\[
\mathcal{A}(:, :, 1) = \begin{bmatrix} 1.1097627 & 1.19273255 \\ 0.13179527 & 0.11751666 \end{bmatrix}, \quad  
\mathcal{A}(:, :, 2) = \begin{bmatrix} 0.13179527 & 0.11751666 \\ 1.10897664 & 1.10577898 \end{bmatrix}
\]  

\[
\mathcal{B}(:, :, 1) = \begin{bmatrix} 2.08473096 & 2.11360891 \\ 0.10834813 & 0.09966327 \end{bmatrix}, \quad  
\mathcal{B}(:, :, 2) = \begin{bmatrix} 0.10834813 & 0.09966327 \\ 2.1783546 & 2.01742586 \end{bmatrix}.
\]  

\end{example}

\noindent By applying \eqref{metric}, we obtain the Bures-Wasserstein distance between Tensors $\mathcal{A}$ and $\mathcal{B}$ is 0.7054867277404278.

The metric $d(\mathcal{A}, \mathcal{B})$ extends naturally from matrices to the tensor space $\mathbb{T}_{psd}(\mathbf{R})$. In its matrix form \cite{bures}, this metric has garnered significant interest across various fields. In quantum information theory \cite{buresquant}, it is known as the Bures distance, while in statistics and optimal transport theory, it is referred to as the Wasserstein metric. 

\section{Computational Considerations}
\label{sec:complexity}

While the theoretical properties of our tensor operations are established in Sections~\ref{results}, it is equally important to understand their computational characteristics. The asymptotic complexity analysis presented here informs the practical applicability of our methods, particularly for large-scale tensor problems common in machine learning and signal processing applications.

\begin{table}[t]
\centering
\caption{Computational complexity of tensor operations}
\label{tab:complexity}
\begin{tabular}{lc}
\hline
Operation & Complexity \\
\hline
T-product & $O(p(mns + mn \log p))$ \\
Eigenvalue bounds & $O(n^3p^3)$ \\
Bures-Wasserstein distance & $O(n^3p)$ \\
Ky Fan sums & $O(n^3p^3)$ \\
\hline
\end{tabular}
\end{table}

The cubic dependence on $n$ in several operations suggests that our methods are particularly well-suited for problems where the tensor order $p$ grows more slowly than the side dimension $n$. 

The complexity analysis yields important insights for real-world implementation of our tensor operations. The T-product's efficient log-linear scaling with tube dimension $p$ makes it particularly well-suited for processing high-order tensors, which frequently arise in applications like video analysis (where $p$ represents temporal frames) and multi-sensor data fusion. This favorable scaling enables practical computation even with thousands of tubes. Meanwhile, the Bures-Wasserstein distance's linear dependence on $p$ and cubic scaling in $n$ permits its use in iterative optimization schemes, as the distance metric can be recomputed efficiently during algorithm convergence.

The computational characteristics also reveal important limitations to consider. The $O(n^3p^3)$ complexity of exact eigenvalue computations becomes prohibitive for tensors with large frontal slice dimensions $(n>100)$, suggesting the need for approximation techniques like randomized SVD or Lanczos methods in such cases. This trade-off between exact computation and approximation accuracy must be carefully balanced based on the specific application requirements and available computational resources. These complexity properties provide clear guidance for selecting appropriate tensor operations based on problem dimensions and accuracy needs.

\section{Conclusion}\label{conclusion}

This article establishes trace-based eigenvalue bounds for positive semi-definite (PSD) tensors under the T-product framework and explores their geometric and algebraic properties. By extending classical distance metrics, such as the Bures-Wasserstein distances, from matrix spaces to tensor spaces, we provide a new perspective on measuring similarities between PSD tensors. \\

\noindent The theoretical results are supported by illustrative examples, demonstrating how the proposed eigenvalue bounds and distance metrics behave when the PSD condition is relaxed. The study concludes with a comprehensive analysis of the computational complexity of the proposed methods, highlighting their efficiency and feasibility.

\section*{Competing Interest}
The authors declare that they have no competing interests.

\section*{Author's contributions}
All authors contributed equally to all.

\section*{Acknowledgment}
Authors would like to thank IIITDM Kancheepuram for the infrastructure and facilities
to carry out this research.

\begin{Backmatter}

\printaddress

\end{Backmatter}

\end{document}